\documentclass[]{amsart}
\usepackage{mathrsfs}
\usepackage{color}
\usepackage{amsmath}
\usepackage{amsfonts}
\usepackage{amssymb}

 \newtheorem{Theorem}{Theorem}[section]
 \newtheorem{Corollary}[Theorem]{Corollary}
 \newtheorem{Lemma}[Theorem]{Lemma}
 \newtheorem{Proposition}[Theorem]{Proposition}

 \newtheorem{Remark}[Theorem]{Remark}

 \numberwithin{equation}{section}

\begin{document}

\title{Twisted version of Strong openness property in $
L^p$}

\author{Qi'an Guan, Zheng Yuan}
\address{Qi'an Guan: School of Mathematical Sciences,
Peking University, Beijing, 100871, China. }
\email{guanqian@math.pku.edu.cn}

\address{Zheng Yuan: School of Mathematical Sciences,
Peking University, Beijing, 100871, China.}
\email{zyuan@pku.edu.cn}

\thanks{}

\subjclass[2010]{32D15 32E10 32L10 32U05 32W05}

\keywords{strong openness property, $L^p$ integral, multiplier ideal sheaf, plurisubharmonic function}

\date{\today}

\dedicatory{}

\commby{}


\begin{abstract}
In this article, we present a twisted version of strong openness property in $L^p$ with applications.
\end{abstract}

\maketitle
\section{Introduction}
The strong openness property is an important feature of multiplier ideal sheaves and used in the study of several complex variables, algebraic geometry and complex geometry 
(see e.g. \cite{GZopen-c,GZ20,berndtsson20,ZhouZhu20,DEL21,ZhouZhu20siu's,FoW20,KS20,ZhouZhu18,DEL18,FoW18,cdM17,cao17,K16}).

Recall that  the multiplier ideal sheaf $\mathcal{I}(\varphi)$ (see  \cite{tian87,Nadel90,siu96,demailly-note2000}) was defined as the sheaf of germs of holomorphic functions $f$ such that
$|f|^{2}e^{-\varphi}$ is locally integrable,
where  $\varphi$ is a plurisubharmonic function (weight) on a complex manifold.
In \cite{GZopen-c}, Guan-Zhou proved  strong openness property for multiplier ideal sheaves
i.e. $\mathcal{I}(\varphi)=\mathcal{I}_{+}(\varphi):=\cup_{q>1}\mathcal{I}(q\varphi)$,
which was a conjecture posed by Demailly (see \cite{demailly2010,demailly-note2000}).
Two dimensional case of the conjecture was proved by Jonsson-Musta\c{t}\u{a} \cite{JM12}.

When $\mathcal{I}(\varphi)=\mathcal{O}$,
Demailly's strong openness conjecture degenerates to the openness conjecture,
which was posed by Demailly-Koll\'{a}r \cite{D-K01} and proved by Berndtsson \cite{berndtsson13} (2-dimensional case was proved by Favre-Jonsson \cite{FM05j}).

Using the strong openness property of multiplier ideal sheaves, Guan-Zhou \cite{GZ-lelong1} gave a characterization of the multiplier ideal sheaves with weights of Lelong number one.
Some generalized versions can be referred to \cite{GZ20}.

Recently, Xu \cite{xu} completed the algebraic approach to the openness conjecture,
which was conjectured by Jonsson-Musta\c{t}\u{a} \cite{JM12}.

\subsection{Background}

Let $\varphi$ be a plurisubharmonic function on a domain $D\subset\mathbb{C}^n$ containing the origin $o$. Let $I$ be an ideal of $\mathcal{O}_{o}$, which is generated by $\{f_j\}_{j=1,...,l}$.
Denote that
$$\log|I|:=\log\max_{1\leq j\leq l}|f_j|,$$
and $c_{o,p}^I(\varphi):=\sup\{c\geq0: |I|^pe^{-2c\varphi}$ is $L^1$ on a neighborhood of $o$$\}$.
When $p=2$,  $c_{o,2}^I(\varphi)$ is the jumping number $c_{o}^I(\varphi)$ (see \cite{JM13}).

In \cite{GZ-soc17}, Guan-Zhou established the following twisted version of strong openness property by using their solutions (\cite{GZopen-effect}) of two conjectures posed by Demailly-Koll\'{a}r and Jonsson-Musta\c{t}\u{a}.
\begin{Theorem}(see \cite{GZ-soc17})\label{thm:1}
	Let $a(t)$ be a positive measurable function on $(-\infty,+\infty)$ such that $a(t)e^{t}$ is strictly increasing and continuous near $+\infty$. Then the following three statements are equivalent:
	
	$(A)$ $a(t)$ is not integrable near $+\infty$;
	
	$(B)$ $a(-2c_{o}^{I}(\varphi)\varphi)\exp(-2c_{o}^{I}(\varphi)\varphi+2\log|I|)$ is not integrable near $o$ for any $\varphi$ and $I$ satisfying $c_{o}^{I}(\varphi)<+\infty$;
	
	$(C)$ $a(-2c_{o}^{I}(\varphi)\varphi+2\log|I|)\exp(-2c_{o}^{I}(\varphi)\varphi+2\log|I|)$ is not integrable near $o$ for any $\varphi$ and $I$ satisfying $c_{o}^{I}(\varphi)<+\infty$.	
\end{Theorem}

In \cite{Fo15}, Forn{\ae}ss established the following strong openness property in $L^p$ by using the strong openness property, which implies the coherence of $L^p$ multiplier ideal sheaves (see \cite{siu04,cao17}):

\emph{Let $F$ be a holomorphic function on a domain $D\subset\mathbb{C}^n$ containing the origin $o$, $\varphi$ a plurisubharmonic function on $D$ and $p\in(0,+\infty)$. If $|F|^pe^{-\varphi}$ is $L^1$ on a neighborhood of $o$, then there exists $q>1$ such that $|F|^pe^{-q\varphi}$ is $L^1$ on a neighborhood of $o$.}

\subsection{Main result and applications}
In the present paper, we give the following twisted version of strong openness property in $L^p$ by using the concavity of minimal $L^2$ integrals and the solutions (\cite{GZopen-effect}) of two conjectures of Demailly-Koll\'{a}r and Jonsson-Musta\c{t}\u{a}.

\begin{Theorem}\label{thm:main}
	Let $p\in(0,+\infty)$, and let $a(t)$ be a positive measurable function on $(-\infty,+\infty)$. If  one of the following conditions holds:
	
	$(1)$ $a(t)$ is decreasing  near $+\infty$;
	
	$(2)$ $a(t)e^t$ is increasing near $+\infty$.
	
	 Then the following three statements are equivalent:
	
	$(A)$ $a(t)$ is not integrable near $+\infty$;
	
	$(B)$ $a(-2c_{o,p}^{I}(\varphi)\varphi)\exp(-2c_{o,p}^{I}(\varphi)\varphi+p\log|I|)$ is not integrable near $o$ for any $\varphi$ and $I$ satisfying $c_{o,p}^{I}(\varphi)<+\infty$;
	
	$(C)$ $a(-2c_{o,p}^{I}(\varphi)\varphi+p\log|I|)\exp(-2c_{o,p}^{I}(\varphi)\varphi+p\log|I|)$ is not integrable near $o$ for any $\varphi$ and $I$ satisfying $c_{o,p}^{I}(\varphi)<+\infty$.	
	\end{Theorem}
\begin{Remark}
	When $p=2$ and condition (2) holds, Theorem \ref{thm:main} is a general version of Theorem \ref{thm:1}.
\end{Remark}

Let $I=\mathcal{O}_{o}$, and denote $c_o^1(\varphi)$ by $c_o(\varphi)$. Theorem \ref{thm:main} implies the following.

\begin{Corollary}\label{c:c}
	Let $\varphi<0$ be a plurisubharmonic function on a neighborhood $U\ni o$ with $c_o^1(\varphi)\in(0,+\infty)$. Let $\kappa:\mathbb{R}\rightarrow(0,+\infty)$ be a decreasing  function satisfying that $\kappa$ is not integrable near $+\infty$. Then for any neighborhood $V\ni o$ one has
	$$|\varphi|^{-1}\kappa(\log(-\varphi))e^{-2c_o(\varphi)\varphi}\not\in L^1(V).$$
\end{Corollary}

When $\kappa\in C^{1}(0,+\infty)$ satisfies $\kappa(t)\geq-5\kappa'(t)$ for $t\gg1$, Corollary \ref{c:c} can be referred to \cite{chen18}.

\begin{Remark}
Corollary \ref{c:c} gives an affirmative answer to a question posed in \cite{chen18}, i.e.,
the condition  $\kappa(t)\geq-5\kappa'(t)$ for $t\gg1$  (Condition (2) of Theorem 1.2 in \cite{chen18}) can be removed.
\end{Remark}

\section{Preparation}

Let $D\subset\mathbb{C}^n$ be a pseudoconvex domain containing the origin $o$,  and let  $\varphi$ and $\psi$ be  plurisubharmonic functions on $D$. Take $T=-\sup_{D}\psi$. Let $f$ be a holomorphic function on a neighborhood of the origin $o$.

We call a positive measurable function $c$ on $(T,+\infty)$ in class $\mathcal{Q}_T$ if the following two statements hold:

$(1)$ $c(t)e^{-t}$ is decreasing with respect to $t$;

$(2)$ $\liminf_{t\rightarrow+\infty}c(t)>0$.

Denote
\begin{equation*}
\inf\{\int_{\{\psi<-t\}}|\tilde{f}|^{2}e^{-\varphi}c(-\psi):(\tilde{f}-f,o)\in
\mathcal{I}(\varphi+\psi)_{o}\,\&{\,}\tilde{f}\in \mathcal{O}(\{\psi<-t\})\},
\end{equation*}
by $G(t)$,  where $t\in[T,+\infty)$ and $c\in\mathcal{Q}_T$.

In \cite{GY-concavity}, we obtain the following concavity of $G(t)$.
\begin{Theorem}
	(see \cite{GY-concavity})\label{thm:concavity}
	If there exists $t\in[T,+\infty)$ satisfying $G(t)<+\infty$, then $G(h^{-1}(r))$ is concave with respect to $r\in(\int_{T_1}^{T}c(t)e^{-t}dt,\int_{T_1}^{+\infty}c(t)e^{-t}dt)$, $\lim_{t\rightarrow T+0}G(t)=G(T)$ and $\lim_{t\rightarrow+\infty}G(t)=0$, where $h(t)=\int_{T_1}^tc(l)e^{-l}dl$ and $T_1\in(T,+\infty)$.
\end{Theorem}
 Theorem \ref{thm:concavity} implies the following corollary, which will be used in the proof of Theorem \ref{thm:main}.
\begin{Corollary}
(see \cite{GY-concavity})
\label{infty}
	Let $c\in\mathcal{Q}_T.$ If $\int_{T_1}^{+\infty}c(t)e^{-t}dt=+\infty$ for some $T_1>T$, and $(f,o)\notin
\mathcal{I}(\varphi+\psi)_{o}$, then $G(t)=+\infty$ for any $t\geq T$, i.e., there is no holomorphic function $\tilde{f}$ on $\{\psi<-t\}$ satisfying $(\tilde{f}-f,o)\in
\mathcal{I}(\varphi+\psi)_{o}$ and $\int_{\{\psi<-t\}}|\tilde{f}|^{2}e^{-\varphi}c(-\psi)<+\infty$.	
\end{Corollary}

Let $p\in(0,+\infty)$. Let $f$ be a holomorphic function on pseudoconvex domain $D\subset\mathbb{C}^{n}$ containing the origin $o$, and let $I$ be an ideal of $\mathcal{O}_o$.  Denote that $C_{f,I,\Phi}(D)=\inf\{\int_{D}|\tilde f|^2e^{-\Phi}:(\tilde f-f,o)\in I\,\&\,\tilde f\in\mathcal{O}(D)\}$, where $\Phi$ is a plurisubharmonic function on $D$.

 Theorem \ref{thm:concavity} implies the following proposition (the case $p=2$ can be referred to \cite{guan_sharp}).
\begin{Proposition}
	\label{p:DK}Let $f$ be a holomorphic function on pseudoconvex domain $D\subset\mathbb{C}^{n}$ containing the origin $o$, and let $\varphi$ be a negative plurisubharmonic function on $D$.
	If $c_{o,p}^f(\varphi)<+\infty$, then
	\begin{equation}
		\label{eq:210820a}
		\frac{1}{r}\int_{\{2c_{o,p}^{f}(\varphi)\varphi<\log r\}}|f|^p\geq C_{f_1,\mathcal{I}(\varphi_1+2c_{o,p}^{f}(\varphi)\varphi)_o,\varphi_1}(D)>0
	\end{equation}
	holds for any $r\in(0,1)$, where $f_1=f^{\lceil\frac{p}{2}\rceil}$ and $\varphi_1=(2\lceil\frac{p}{2}\rceil-p)\log|f|$.
	\end{Proposition}
\begin{proof}
	It suffices to prove the case $\int_{\{2c_{o,p}^{f}(\varphi)\varphi<\log r\}}|f|^p<+\infty$.   Theorem \ref{thm:concavity} ($\varphi\sim\varphi_1$, $\psi\sim 2c_{o,p}^{f}(\varphi)\varphi$, $f\sim f_1$ and $c\sim1$, here $\sim$ means the former replaced by the latter) implies that
	\begin{equation}
		\label{eq:210820g}\begin{split}
		\int_{\{2c_{o,p}^{f}(\varphi)\varphi<-t\}}|f_1|^2e^{-\varphi_1}&\geq G(t)\\
		&\geq e^{-t}G(0)\\
		&=e^{-t} C_{f_1,\mathcal{I}(\varphi_1+2c_{o,p}^{f}(\varphi)\varphi)_o,\varphi_1}
		\end{split}	\end{equation}
		holds for any $t\geq0$.
		Taking $t=-\log r$, inequality \eqref{eq:210820g} becomes
		$$\frac{1}{r}\int_{\{2c_{o,p}^{f}(\varphi)\varphi<\log r\}}|f|^p\geq C_{f_1,\mathcal{I}(\varphi_1+2c_{o,p}^{f}(\varphi)\varphi)_o,\varphi_1}(D).$$
		
In the following, we prove $C_{f_1,\mathcal{I}(\varphi_1+2c_{o,p}^{f}(\varphi)\varphi)_o,\varphi_1}(D)>0$ by contradiction: 
if $C_{f_1,\mathcal{I}(\varphi_1+2c_{o,p}^{f}(\varphi)\varphi)_o,\varphi_1}(D)=0,$ 
there exist holomorphic function $\{F_j\}_{j=1,2,...}$ on $D$, such that $(F_j-f_1,o)\in\mathcal{I}(\varphi_1+2c_{o,p}^{f}(\varphi)\varphi)_o$ and $\lim_{j\rightarrow+\infty}\int_{D}|F_j|^2e^{-\varphi_1}=0$. 
As $e^{-\varphi_1}$ has locally positive lower bound on $D$,  there exists a subsequence of $\{F_j\}$  compactly convergent to $0$. 
It follows from the closedness of $\mathcal{I}(\varphi_1+2c_{o,p}^{f}(\varphi)\varphi)_o$ under the topology of compact convergence (see \cite{G-R}) that $(f_1,o)\in\mathcal{I}(\varphi_1+2c_{o,p}^{f}(\varphi)\varphi)_o$, 
which contradicts that $|f|^pe^{-2c_{o,p}^f(\varphi)\varphi}$ is not integrable near $o$. Hence, we obtain $C_{f_1,\mathcal{I}(\varphi_1+2c_{o,p}^{f}(\varphi)\varphi)_o,\varphi_1}(D)>0.$
\end{proof}
Choosing $(f,o)\in I\subseteq \mathcal{O}_{o}$ such that $c_{o,p}^I(\varphi)=c_{o,p}^f(\varphi)$, Proposition \ref{p:DK} implies the following result.
\begin{Corollary}\label{c:DK}
	Let $p\in(0,+\infty)$, and let $I$ be an ideal of $\mathcal{O}_o$. Let $\varphi$ be a negative plurisubharmonic function on $D$.
	If $c_{o,p}^I(\varphi)<+\infty$, then
	\begin{equation}
		\label{eq:210820b}
		\frac{1}{r}\int_{\{2c_{o,p}^{I}(\varphi)\varphi<\log r\}}|I|^p	\end{equation}has positive lower bounds independent of $r\in(0,1)$.
		\end{Corollary}
The following Lemma (see \cite{GZopen-effect}) implies a solution of a conjecture posed by Jonsson and Musta\c{t}\u{a}. 	
\begin{Lemma}(see \cite{GZopen-effect})
	\label{l:JM0}
	Let $B\in(0,1]$. Let $\varphi$ be a negative plurisubharmonic function on pseudoconvex domain $D\subset\mathbb{C}^n$. Let $f$ be a bounded holomorphic function on $D$. Assume that $|f|^2e^{-\varphi}$ is not locally integrable near $z_0\in D$. Then we obtain that
	\begin{equation}
		\label{eq:210820d}
		\liminf_{R\rightarrow+\infty}e^{R}\frac{1}{B}\mu(\{-R-B<\varphi-2\log|f|<-R\})\geq\frac{C}{2e^{B}},
	\end{equation}
	where $C>0$ is a constant independent of $B$ and $\mu$ is the Lebesgue measure on $\mathbb{C}^n$.
\end{Lemma}
		Taking $R=kB$ in inequality \eqref{eq:210820d}, for any given $\epsilon>0$, there exists $k_0$ depending on $B$, such that for any $k\geq k_0$, one can obtain that
		$$e^{(k+1)B}\frac{1}{B}\mu(\{-(k+1)B<\varphi-2\log|f|<-kB\})\geq \frac{C-\epsilon}{2},$$
		i.e.
		\begin{equation}
			\label{eq:210820e}\mu(\{-(k+1)B<\varphi-2\log|f|<-kB\})\geq e^{-(k+1)B}B\frac{C-\epsilon}{2}.		\end{equation}
Take sum $k\geq k_0$ in inequality \eqref{eq:210820e}, and let $B$ go to $0$, one can obtain that
\begin{equation}
	\label{eq:210820f}\liminf_{R\rightarrow+\infty}e^{R}\mu(\{\varphi-2\log|f|<-R\})\geq \frac{C}{2}.
\end{equation}		
		
Inequality \eqref{eq:210820f} implies the following Proposition, 
which will be used in the proof of Theorem \ref{thm:main}. 
\begin{Proposition}
	\label{p:JM}(see \cite{GZopen-effect})
	Let $p\in(0,+\infty)$, and let $I$ be an ideal of $\mathcal{O}_o$. 
Let $\varphi$ be a negative plurisubharmonic function on pseudoconvex domain $D$ (the origin $o\in D$).
	If $c_{o,p}^I(\varphi)<+\infty$, then
	\begin{equation}
		\label{eq:210820c}
		\frac{1}{r}\mu(\{2c_{o,p}^{I}(\varphi)\varphi-p\log|I|<\log r\})
\end{equation}has positive lower bounds independent of $r\in(0,1)$.
\end{Proposition}
\begin{proof}

By the strong openness property, 
we have $|I|^pe^{-2c_{o,p}^I(\varphi)\varphi}$ is not integrable on any neighborhood of $o$. 
There exists $(f,o)\in I$ such that $|f|^pe^{-2c_{o,p}^I(\varphi)\varphi}$ is not integrable on any neighborhood of $o$.

Denote that $f_1=f^{\lceil\frac{p}{2}\rceil}$ and $\varphi_1=(2\lceil\frac{p}{2}\rceil-p)\log|f|$, 
then $|f_1|^2e^{-\varphi_1-2c_{o,p}^I(\varphi)\varphi}$ is not integrable on any neighborhood of $o$.  
Replacing $\varphi$ by $2c_{o,p}^I(\varphi)\varphi+\varphi_1$, and 
$f$ by $f_1$ and $R$ by $-\log r$ in inequality \eqref{eq:210820f}, we obtain that
 \begin{displaymath}
 	\begin{split}
 		 &\liminf_{r\rightarrow0}\frac{1}{r}\mu(\{2c_{o,p}^I(\varphi)\varphi-p\log|f|<\log r\})\\
 		 =&\liminf_{r\rightarrow0}\frac{1}{r}\mu(\{2c_{o,p}^I(\varphi)\varphi+\varphi_1-2\log|f_1|<\log r\})\\
 		 >&0,
 		  	\end{split}
 \end{displaymath}
which implies that $\mu(\{2c_{o,p}^{I}(\varphi)\varphi-p\log|I|<\log r\})$ has positive lower bounds independent of $r\in(0,1)$.
\end{proof}

The following two lemmas will be used to prove Theorem \ref{thm:main}.
\begin{Lemma}\label{l:2}
Let $a(t)$ be a positive measurable function on $(-\infty,+\infty)$, 
such that $a(t)e^{t}$ is increasing near $+\infty$, 
and $a(t)$ is not integrable near $+\infty$.  
Then there exists a positive measurable function $\tilde a(t)$ on $(-\infty,+\infty)$ satisfying the following statements:
		
		$(1)$ there exists $T<+\infty$ such that $\tilde{a}(t)\leq a(t)$ for any $t>T$;
		
		$(2)$  $\tilde a(t)e^{t}$ is strictly increasing and continuous near $+\infty$;
		
		$(3)$ $\tilde a(t)$ is not integrable near $+\infty$.
\end{Lemma}
\begin{proof}
	There exists $T\in\mathbb{Z}$ such that $a(t)e^{t}$ is increasing on $(T-1,+\infty)$. As $a(t)$ is not integrable near $+\infty$ and $a(t)e^{t}$ is increasing on $(T-1,+\infty)$, we have $\sum_{n=T-1}^{+\infty}a(n)=+\infty$ and $a(n)\leq a(n+1)e$.
Then there exists a sequence of real numbers $\{b_n\}_{n=T}^{+\infty}$ satisfying that: $(1)$ $\sum_{n=T}^{+\infty}b_n=+\infty$; $(2)$ $b_{n+1}\leq a(n)$ for any $n\geq T-1$; $b_n<b_{n+1}e$ for any $n\geq T$.

Take $$\tilde a(t)=\left\{\begin{array}{lcl}
	
		\frac{b_n}{e^2}(\frac{b_{n+1}}{b_n})^{t-n} & \mbox{if} & t\in[n,n+1)\subset[T,+\infty),\\
		1 &\mbox{if} & t\in(-\infty,T).
		\end{array}
	\right.$$
Now, we prove that $\tilde a$ satisfies the three statements in Lemma \ref{l:2}.  Following from $b_n<b_{n+1}e$, $b_{n+1}\leq a(n)$ for any $n\geq T-1$ and $a(t)e^{t}$ is increasing on $(T-1,+\infty)$, we obtain that $\tilde a(t)\leq\frac{1}{e^2}\max{\{b_n,b_{n+1}\}}\leq \frac{b_{n+1}}{e}\leq\frac{a(n)}{e}\leq a(t)$  for any $t\in[n,n+1)\subset[T,+\infty)$. As $b_n<b_{n+1}e$, then $\tilde a(t)e^{t}$ is strictly increasing  near $+\infty$. The continuity of $\tilde a(t)$ on $(T,+\infty)$ is  just from the construction of $\tilde a(t)$. Note that
\begin{displaymath}
	\int_{N}^{+\infty}\tilde a(t)dt\geq\sum_{n=N}^{+\infty}\frac{1}{e^2}\min{\{b_n,b_{n+1}\}}\geq\sum_{n=N}^{+\infty}\frac{b_{n}}{e^3}=+\infty
\end{displaymath}
holds for any integer $N\geq T$.

Thus, Lemma \ref{l:2} holds.
\end{proof}

\begin{Lemma}
	\label{l:m}(see \cite{GZ-soc17}) For any two measurable spaces $(X_i,\mu_i)$ and two measurable functions $g_i$ on $X_i$, respectively ($i\in\{1,2\}$), if $\mu_1(\{g_1\geq r^{-1}\})\geq\mu_2(\{g_2\geq r^{-1}\})$ for any $r\in(0,r_0]$, then $\int_{\{g_1\geq r_0^{-1}\}}g_1d\mu_1\geq\int_{\{g_2\geq r_0^{-1}\}}g_2d\mu_2$.
\end{Lemma}

\section{Proofs of Theorem \ref{thm:main} and corollary \ref{c:c}}

In this section, we prove Theorem \ref{thm:main} and Corollary \ref{c:c}.
\begin{proof}[Proof of Theorem \ref{thm:main}]
 We prove Theorem \ref{thm:main} in two cases, that $a(t)$ satisfies condition $(1)$ or condition $(2)$.

\

\emph{Case $(1)$. $a(t)$ is decreasing near $+\infty$.}

\

Firstly, we prove $(B)\Rightarrow(A)$ and $(C)\Rightarrow(A)$.
Consider $I=1$ and $\varphi=\log|z_1|$ on the unit polydisc $\Delta^n\subset\mathbb{C}^n$. Note that $c_{o,p}^1(\log|z_1|)=1$ and
\begin{displaymath}
	\begin{split}
		\int_{\Delta_{r_0}^n}a(-2\log|z_1|)\frac{1}{|z_1|^2}=&(\pi r_0^2)^{n-1}\int_{\Delta_{r_0}}a(-2\log|z_1|)\frac{1}{|z_1|^2}\\
		=&(\pi r_0^2)^{n-1}2\pi\int_0^{r_0}a(-2\log r)r^{-1}dr\\
		=&(\pi r_0^2)^{n-1}\pi\int_{-2\log{r_0}}^{+\infty}a(t)dt,
	\end{split}
\end{displaymath}
then we obtain $(B)\Rightarrow (A)$ and $(C)\Rightarrow (A)$.

Then, we prove $(A)\Rightarrow(B)$.
The strong openness property shows that there exists $(f,o)\in I$ such that $|f|^pe^{-2c_{o,p}^I(\varphi)\varphi}$ is not integrable near $o$. It suffices to prove that $|f|^pe^{-2c_{o,p}^I(\varphi)\varphi}a(-2c_{o,p}^I(\varphi)\varphi)$ is not integrable near $o$.

 Denote that $f_1=f^{\lceil \frac{p}{2}\rceil}$ and $\varphi_1=(2\lceil \frac{p}{2}\rceil-p)\log|f|$, where $\lceil m\rceil:=\min\{n\in\mathbb{Z}:n\geq m\}$. Note that $|f|^pe^{-2c_{o,p}^I(\varphi)\varphi}=|f_1|^2e^{-\varphi_1-2c_{o,p}^I(\varphi)\varphi}$ and it's not integrable near $o$.  Then we assume that $|f_1|^2e^{-\varphi_1-2c_{o,p}^I(\varphi)\varphi}a(-2c_{o,p}^I(\varphi)\varphi)$ is  integrable near $o$ to get a contradiction. Set $c(t)=a(t)e^{t}+1$. There exists a pseudoconvex domain $D\subset\mathbb{C}^n$ containing the origin $o$ such that $\int_D|f_1|^2e^{-\varphi_1-2c_{o,p}^I(\varphi)\varphi}a(-2c_{o,p}^I(\varphi)\varphi)<+\infty$, $\int_D|f_1|^2e^{-\varphi_1}=\int_D|f|^p<+\infty$  and $c(t)e^{-t}=a(t)+e^{-t}$ is decreasing on $(T,+\infty)$, where $T=-\sup_{D}2c_{o,p}^I(\varphi)\varphi$. Note that $\liminf_{t\rightarrow+\infty}c(t)>0$, then $c\in\mathcal{Q}_T$. As $a(t)$ is not integrable near $+\infty$, so is $c(t)e^{-t}$. Using Corollary \ref{infty} ($f$, $\varphi$ and $\psi$ are replaced by $f_1$, $\varphi_1$ and $2c_{o,p}^I(\varphi)\varphi$, respectively), as $(f_1,o)\not\in\mathcal{I}(\varphi_1+2c_{o,p}^I(\varphi)\varphi)_o$, then we have $G(T)=+\infty$. By the definition of $G(T)$, we obtain that
 \begin{displaymath}
 	\begin{split}
 	G(T)&\leq\int_D|f_1|^2e^{-\varphi_1}c(-2c_{o,p}^I(\varphi)\varphi)\\
 	&=\int_D|f_1|^2e^{-\varphi_1-2c_{o,p}^I(\varphi)\varphi}a(-2c_{o,p}^I(\varphi)\varphi)+\int_D|f_1|^2e^{-\varphi_1}\\
 	&<+\infty
 		 	\end{split}
 \end{displaymath}
which contradicts to $G(T)=+\infty$. Thus, we obtain $(A)\Rightarrow (B)$.

Finally, we prove $(A)\Rightarrow (C)$.
If $I=1$, following from the above discussion, we have $(C)$ holds. Thus, it is suffices to consider the case $I\not=1$.

If $a(t)$ is decreasing near $+\infty$, we have $a(-2c_{o,p}^{I}(\varphi)\varphi)\exp(-2c_{o,p}^{I}(\varphi)\varphi+p\log|I|)$ is not integrable near $o$ for any plurisubharmonic function $\varphi$ on $D$ and $I$ satisfying $c_{o,p}^{I}(\varphi)<+\infty$. Note that there exists a small neighborhood $D'$ of $o$ such that $\log|I|<0$ on $D'$, then $a(t)$ is decreasing near $+\infty$ implies that
$$a(-2c_{o,p}^{I}(\varphi)\varphi+p\log|I|)\geq a(-2c_{o,p}^{I}(\varphi)\varphi).$$
Thus we obtain that $a(-2c_{o,p}^{I}(\varphi)\varphi+p\log|I|)\exp(-2c_{o,p}^{I}(\varphi)\varphi+p\log|I|)$ is not integrable near $o$.

Thus, we prove Theorem \ref{thm:main} for the case that $a(t)$ satisfies condition $(1)$.

\

\emph{Case $(2)$. $a(t)e^{t}$ is increasing near $+\infty$.}

\

In this case, the proofs of $(B)\Rightarrow(A)$ and $(C)\Rightarrow(A)$ are the same as the case $(1)$, therefore it suffices to prove $(A)\Rightarrow(B)$ and $(A)\Rightarrow(C)$.

Assume that statement $(A)$ holds.
 It follows from Lemma \ref{l:2} that there exists a positive function $\tilde a(t)$ on $(-\infty,+\infty)$ satisfying that: $\tilde a(t)\leq a(t)$ near $+\infty$; $\tilde a(t)e^{t}$ is strictly increasing and continuous near $+\infty$; $\tilde a(t)$ is not integrable near $+\infty$. Thus, it suffices to prove that $\tilde a(-2c_{o,p}^I(\varphi)\varphi)\exp(-2c_{o,p}^I(\varphi)\varphi+p\log|I|)$ and $\tilde a(-2c_{o,p}^I(\varphi)\varphi+p\log|I|)\exp(-2c_{o,p}^I(\varphi)\varphi+p\log|I|)$ are not integrable near $o$ for any $\varphi$ and $I$ satisfying $c_{o,p}^I(\varphi)<+\infty.$

Fitstly, we prove $\tilde a(-2c_{o,p}^I(\varphi)\varphi)\exp(-2c_{o,p}^I(\varphi)\varphi+p\log|I|)$ is not integrable near $o$ for any $\varphi$ and $I$ satisfying $c_{o,p}^I(\varphi)<+\infty$ by using Corollary \ref{c:DK} and Lemma \ref{l:m}.

Let $X_1$ be a small neighborhood of $o$, and let $X_2=(0,1]$. Let $\mu_1(\cdot)=\int_{\cdot}|I|^2$,  and let $\mu_2$ be the Lebeague measure on $X_2$. Denote that $Y_r=\{-2c_{o,p}^I(\varphi)\varphi\geq -\log r\}$. Corollary \ref{c:DK} shows that there exists a positive constant $C$ such that $\mu_1(Y_r)\geq Cr$ holds for any $r\in(0,1]$.

Let $g_1=\tilde a(-2c_{o,p}^I(\varphi)\varphi)\exp(-2c_{o,p}^I(\varphi)\varphi)$ and $g_2(x)=\tilde a(-\log x+\log C)Cx^{-1}$. As $\tilde a(t)e^{-t}$ is increasing near $+\infty$, then $g_1\geq \tilde a(-\log r)r^{-1}$ on $Y_r$ implies that
\begin{equation}
	\label{eq:210820h}
	\mu_1(\{g_1\geq \tilde a(-\log r)r^{-1}\})\geq\mu_1(Y_r)\geq Cr
\end{equation}
holds for any  $r>0$ small enough.

As $\tilde a(t)e^{t}$ is strictly increasing near $+\infty$, then there exists $r_0\in(0,1)$ such that
\begin{equation}
	\label{eq:210820i}
	\mu_2(\{x\in(0,r_0]:g_2(x)\geq \tilde a(-\log r)r^{-1}\})=\mu_2(\{0<x\leq Cr\})=C r
\end{equation}
for any $r\in(0,r_0]$.

Using the continuity of $\tilde a(-\log r)r^{-1}$ and $\tilde a(-\log r)r^{-1}$ converges to $+\infty$ ( when $r\rightarrow0+0$), we obtain that
$$\mu_1(\{g_1\geq r^{-1}\})\geq\mu_2(\{x\in(0,r_0]:g_2(x)\geq r^{-1}\})$$
holds for any $r>0$ small enough. Following from Lemma \ref{l:m} and $\tilde a(t)$ is not integrable near $+\infty$, we obtain $\tilde a(-2c_{o,p}^I(\varphi)\varphi)\exp(-2c_{o,p}^I(\varphi)\varphi+p\log|I|)$ is not integrable near $o$.

Then,  we prove $\tilde a(-2c_{o,p}^I(\varphi)\varphi+p\log|I|)\exp(-2c_{o,p}^I(\varphi)\varphi+p\log|I|)$ is not integrable near $o$ for any $\varphi$ and $I$ satisfying $c_{o,p}^I(\varphi)<+\infty$ by using Proposition \ref{p:JM} and Lemma \ref{l:m}.

Let $X_1$ be a small neighborhood of $o$, and let $X_2=(0,1]$. Let $\mu_1$ and $\mu_2$ be the Lebeague measure on $X_1$ and $X_2$, respectively. Let $Y_r=\{-2c_{o,p}^I(\varphi)\varphi+p\log|I|\geq -\log r\}$. Proposition \ref{p:JM} shows that there exists a positive constant $C$ such that $\mu_1(Y_r)\geq Cr$ holds for any $r\in(0,1]$.

Let $g_1=\tilde a(-2c_{o,p}^I(\varphi)\varphi+p\log|I|)\exp(-2c_{o,p}^I(\varphi)\varphi+p\log|I|)$ and $g_2(x)=\tilde a(-\log x+\log C)Cx^{-1}$. As $\tilde a(t)e^{-t}$ is increasing near $+\infty$, then $g_1\geq \tilde a(-\log r)r^{-1}$ on $Y_r$ implies that
\begin{equation}
	\label{eq:210820j}
	\mu_1(\{g_1\geq \tilde a(-\log r)r^{-1}\})\geq\mu_1(Y_r)\geq Cr
\end{equation}
holds for any $r>0$ small enough.

As $\tilde a(t)e^{t}$ is strictly increasing near $+\infty$, then there exists $r_0\in(0,1)$ such that
\begin{equation}
	\label{eq:210820k}
	\mu_2(\{x\in(0,r_0]:g_2(x)\geq \tilde a(-\log r)r^{-1}\})=\mu_2(\{0<x\leq Cr\})=C r
\end{equation}
for any $r\in(0,r_0]$.

Using the continuity of $\tilde a(-\log r)r^{-1}$ and $\tilde a(-\log r)r^{-1}$ converges to $+\infty$ ( when $r\rightarrow0+0$), we obtain that
$$\mu_1(\{g_1\geq r^{-1}\})\geq\mu_2(\{x\in(0,r_0]:g_2(x)\geq r^{-1}\})$$
holds for any $r>0$ small enough. Following from Lemma \ref{l:m} and $\tilde a(t)$ is not integrable near $+\infty$, we obtain $\tilde a(-2c_{o,p}^I(\varphi)\varphi+p\log|I|)\exp(-2c_{o,p}^I(\varphi)\varphi+p\log|I|)$ is not integrable near $o$.

Thus, we prove Theorem \ref{thm:main} for the case that $a(t)$ satisfies condition $(2)$.	
\end{proof}

\begin{proof}[Proof of Corollary \ref{c:c}]
	Let $a(t)=\frac{1}{t}\kappa({\log t-\log (2c_o(\varphi)}))$. Note that
$$\int_{N}^{+\infty}a(t)dt=\int_N^{+\infty}\frac{1}{t}\kappa(\log t-\log (2c_o(\varphi)))dt=\int_{\log N}^{+\infty}\kappa(t-\log (2c_{o}(\varphi)))dt=+\infty$$
for $N\gg 1$.
Then the case $p=2$, $I=1$ and $a(t)$ satisfying condition $(1)$ of Theorem \ref{thm:main} implies that Corollary \ref{c:c} holds.
\end{proof}

\vspace{.1in} {\em Acknowledgements}.  The first named author was supported by NSFC-11825101, NSFC-11522101 and NSFC-11431013.

\bibliographystyle{references}
\bibliography{xbib}

\end{document}